\documentclass{amsart}

\usepackage[T1]{fontenc}
\usepackage{amssymb}
\usepackage{enumitem}
\usepackage{mathrsfs}
\usepackage[colorlinks, linkcolor=blue, citecolor=blue, urlcolor=blue]{hyperref}
\usepackage{tikz}

\newtheorem{theorem}{Theorem}[section]
\newtheorem{corollary}[theorem]{Corollary}
\newtheorem{lemma}[theorem]{Lemma}

\theoremstyle{definition}

\newtheorem{question}[theorem]{Question}

\numberwithin{equation}{section}

\DeclareMathOperator{\ran}{ran}
\DeclareMathOperator{\fin}{fin}
\DeclareMathOperator{\spn}{span}
\DeclareMathOperator{\Sub}{Sub}
\DeclareMathOperator{\cl}{cl}

\title{Amorphous sets and dual Dedekind finiteness}

\author{Yifan Hu}
\address{School of Philosophy\\
Wuhan University\\
No.~299 Bayi Road\\
Wuhan\\
Hubei Province 430072\\
People's Republic of China}
\email{djfrankyifanhu@outlook.com}

\author{Ruihuan Mao}
\address{School of Mathematical Sciences and LPMC\\
Nankai University\\
No.~94 Weijin Road\\
Tianjin 300071\\
People's Republic of China}
\email{rhmao2003@outlook.com}

\author{Guozhen Shen}
\address{Department of Philosophy (Zhuhai)\\
Sun Yat-sen University\\
No.~2 Daxue Road\\
Zhuhai\\
Guangdong Province 519082\\
People's Republic of China}
\email{shen\_guozhen@outlook.com}



\begin{document}

\begin{abstract}
A set $A$ is dually Dedekind finite if every surjection from $A$ onto $A$ is injective; otherwise, $A$ is dually Dedekind infinite. An amorphous set is an infinite set that cannot be partitioned into two infinite subsets. A strictly amorphous set is an amorphous set in which every partition has only finitely many non-singleton blocks. It is proved consistent with $\mathsf{ZF}$ (i.e., the Zermelo--Fraenkel set theory without the axiom of choice) that there exists an amorphous set $A$ whose power set $\mathscr{P}(A)$ is dually Dedekind infinite, which gives a negative solution to a question proposed by Truss [J. Truss, \emph{Fund. Math.} 84, 187--208 (1974)]. Nevertheless, we prove in $\mathsf{ZF}$ that, for all strictly amorphous sets $A$ and all natural numbers $n$, $\mathscr{P}(A)^n$ is dually Dedekind finite, which generalizes a result of Goldstern.
\end{abstract}

\maketitle

\section{Introduction}
In 1888, Dedekind~\cite{Dedekind1888} defined an infinite set as one that is equinumerous with a proper subset of itself.
Such sets are now referred to as \emph{Dedekind infinite}.
Sets that are not Dedekind infinite are called \emph{Dedekind finite}.
In the presence of the axiom of choice, a Dedekind finite set is the same as a finite set,
where a finite set is defined as one that is equinumerous with some natural number.
However, without the axiom of choice, there may exist a Dedekind-finite infinite set.

In 1958, Levy~\cite{Levy1958} investigated the relationships between various notions of finiteness in the absence of the axiom of choice.
An \emph{amorphous set} is an infinite set that cannot be partitioned into two infinite subsets.
Levy referred to amorphous sets and finite sets collectively as Ia-finite sets.

A set $A$ is \emph{dually Dedekind finite} if every surjection from $A$ onto $A$ is injective; otherwise, $A$ is \emph{dually Dedekind infinite}.
The notion of dual Dedekind finiteness was first introduced by Truss~\cite{Truss1974},
who used $\Delta_5$ to denote the class of cardinalities of all dually Dedekind finite sets.
It turns out that dual Dedekind finiteness is a particularly interesting notion of finiteness,
which has been extensively investigated in~\cite{Degen1994,Goldstern1997,Herrlich2015,Mao2025,Panasawatwong2025,Spivsiak1993}.

We use $\mathscr{P}(A)$ and $\fin(A)$ to denote, respectively, the power set and the set of all finite subsets of $A$.
Since the class of dually Dedekind finite sets is closed under finite unions (see~\cite[Theorem~1(iv)]{Truss1974}),
it follows that for any amorphous set $A$, $\mathscr{P}(A)$ is dually Dedekind finite if and only if $\fin(A)$ is
(see also~\cite[Theorem~2.15(3)]{Herrlich2015}).

In the paragraph preceding Theorem~4 of~\cite{Truss1974},
Truss proposed the conjecture that for all amorphous sets $A$, $\fin(A)$ is dually Dedekind finite.
This question was also proposed in \cite[Problem~4.1]{Spivsiak1993}, \cite[Question~1]{Herrlich2015} and \cite[Question~4.8]{Shen2021}.
In this article, we provide a negative solution to the above question.

\begin{theorem}\label{hmsmain}
It is consistent with $\mathsf{ZF}$ that there exists an amorphous set $A$
such that both $\mathscr{P}(A)$ and $\fin(A)$ are dually Dedekind infinite.
\end{theorem}

Truss~\cite{Truss1995} defined a \emph{strictly amorphous set} as an amorphous set
in which every partition has only finitely many non-singleton blocks.
Goldstern~\cite{Goldstern1997} defined a \emph{strongly amorphous set} as an infinite set $A$ such that, for every positive integer $k$,
every $k$-ary relation on $A$ is first-order definable from finitely many parameters in the language of equality.
He also proved that for every strongly amorphous set $A$, $\mathscr{P}(A)$ is dually Dedekind finite.
It is easy to see that every strongly amorphous set is strictly amorphous.
In this article, using ideas in~\cite{Truss1995}, we show that the converse also holds.
Furthermore, we extend Goldstern's result by proving that, for every strictly amorphous set $A$ and every natural number $n$,
$\mathscr{P}(A)^n$ is dually Dedekind finite. This result also provides a partial answer to \cite[Question~4.8]{Shen2021}.

\section{An amorphous set with a dually Dedekind infinite power set}
We shall employ the method of permutation models.
We refer the reader to~\cite[Chap.~8]{Halbeisen2025} or~\cite[Chap.~4]{Jech1973}
for an introduction to the theory of permutation models.
Permutation models are not models of $\mathsf{ZF}$;
they are models of $\mathsf{ZFA}$ (the Zermelo--Fraenkel set theory with atoms).
We shall construct a permutation model in which the atoms form an amorphous set whose power set is dually Dedekind infinite.
Then, by the Jech--Sochor embedding theorem (see~\cite[Theorem~17.1]{Halbeisen2025} or~\cite[Theorem~6.1]{Jech1973}),
we obtain Theorem~\ref{hmsmain}.

We take the set of atoms to be
\[
A=\{a_v\mid v\in V\},
\]
where $V$ is an infinite vector space over the field $\mathbb{F}_2$.
For simplicity, we identify each $a_v$ with $v$, and hence identify $A$ with $V$.
Let $\mathcal{G}$ be the general linear group of $A$; that is, the group of all invertible linear transformations of $A$.

Every permutation $\pi$ of $A$ extents recursively to a permutation of the universe by
\[
\pi x=\{\pi z\mid z\in x\}.
\]
Then $x$ belongs to the permutation model $\mathcal{V}$ determined by $\mathcal{G}$ if and only if
$x\subseteq\mathcal{V}$ and $x$ has a \emph{finite support}, that is, a finite subset $E\subseteq A$
such that every permutation $\pi\in\mathcal{G}$ fixing $E$ pointwise also fixes $x$.

For $S\subseteq A$, we write $\spn(S)$ for the linear span of $S$.
We also write $0$ for the zero vector and $\Sub(A)$ for the set of all linear subspaces of $A$.

\begin{lemma}\label{sh21}
In $\mathcal{V}$, $A$ is amorphous.
\end{lemma}
\begin{proof}
Let $B\in\mathcal{V}$ be a subset of $A$ and let $E\subseteq A$ be a finite support of $B$.
Clearly, for any distinct $u,v\in A\setminus\spn(E)$, there exists a $\pi\in\mathcal{G}$ that fixes $\spn(E)$ pointwise and moves $u$ to $v$,
and thus $u\in B$ if and only if $v\in B$, since $E$ supports $B$. Hence, $B\subseteq\spn(E)$ or $A\setminus B\subseteq\spn(E)$,
which implies that $B$ is finite or cofinite.
\end{proof}

\begin{lemma}\label{sh22}
In $\mathcal{V}$, $\fin(A)$, and consequently $\mathscr{P}(A)$, is dually Dedekind infinite.
\end{lemma}
\begin{proof}
In $\mathcal{V}$, we construct a non-injective surjection $f:\fin(A)\to\fin(A)$ as follows.
For each $S\in\fin(A)$, define
\[
f(S)=
\begin{cases}
S\setminus\bigcup\{W\in\Sub(A)\mid W\subseteq S\text{ with }|W|\text{ maximal}\} & \text{if $0\in S$,}\\
S\cup\{0\}                                                                       & \text{otherwise.}
\end{cases}
\]
Since every $\pi\in\mathcal{G}$ is a linear transformation, it follows that every $\pi\in\mathcal{G}$ fixes $f$,
which implies $f\in\mathcal{V}$. $f$ is not injective since $f(W)=\varnothing$ for all finite $W\in\Sub(A)$.

We show that $f$ is surjective as follows. Take an arbitrary $T\in\fin(A)$.
If $0\in T$, then $T=f(T\setminus\{0\})\in\ran(f)$.
Suppose $0\notin T$. We need to construct an $S\in\fin(A)$ such that $f(S)=T$.
Let $n=|T|$. We define by recursion a sequence $\langle u_i\rangle_{i\leqslant n}$ of vectors as follows.
Let $j\leqslant n$ and assume that $\langle u_i\rangle_{i<j}$ has already been defined.
Since $A$ is infinite, we can choose
\[
u_j\in A\setminus\spn(T\cup\{u_i\mid i<j\}).
\]
Let $U=\spn(\{u_i\mid i\leqslant n\})$ and let $S=T\cup U$.
By construction, it is easy to see that $\spn(T)\cap U=\{0\}$,
which implies that for every $W\in\Sub(A)$ with $W\subseteq S$, either $W\subseteq T\cup\{0\}$ or $W\subseteq U$.
Since $|T\cup\{0\}|=n+1<2^{n+1}=|U|$, it follows from the definition of $f$ that $f(S)=T$.
\end{proof}

Now, Theorem~\ref{hmsmain} follows immediately from Lemmas~\ref{sh21} and~\ref{sh22} together with the Jech--Sochor embedding theorem.

\section{Amorphous sets of projective type and dual Dedekind finiteness}
In this section, we extend the result of the previous section by showing that
every amorphous set of projective type has a dually Dedekind infinite power set.
We first recall the definitions (see~\cite[p.~198]{Truss1995} and~\cite[Section~4.6]{Hodges1993}).

A \emph{pregeometry} on a set $A$ is a function $\cl:\fin(A)\to\fin(A)$ satisfying the following conditions:
\begin{enumerate}[label=\upshape(\arabic*)]
  \item $\cl$ is a closure operator; that is, for all $S,T\in\fin(A)$, $S\subseteq\cl(S)=\cl(\cl(S))$,
        and if $S\subseteq T$ then $\cl(S)\subseteq\cl(T)$. $S\in\fin(A)$ is called \emph{closed} if $\cl(S)=S$.
  \item $\cl$ has the exchange property; that is, for all $S\in\fin(A)$ and $a,b\in A\setminus\cl(S)$,
        $a\in\cl(S\cup\{b\})$ if and only if $b\in\cl(S\cup\{a\})$.
  \item $\cl$ is locally homogeneous; that is, for all closed $S,T\in\fin(A)$ with $S\subseteq T$ and all $a,b\in T\setminus S$,
        there is a permutation $\pi$ of $T$ that preserves $\cl$ on $\fin(T)$, fixes $S$ pointwise, takes $a$ to $b$,
        and is such that, for all closed $U\in\fin(A)$ with $T\subseteq U$, $\pi$ can be extended to $U$ so as to preserve $\cl$ on $\fin(U)$.
\end{enumerate}
If in addition $\cl(\varnothing)=\varnothing$ and $\cl(\{a\})=\{a\}$ for all $a\in A$, then the pregeometry is called a \emph{geometry}.
We say that $\cl$ is \emph{degenerate} if $\cl(S)=\bigcup_{a\in S}\cl(\{a\})$ for all non-empty $S\in\fin(A)$.
An amorphous set is said to be of \emph{projective type} if it admits a non-degenerate pregeometry.
Note that if $A$ is an infinite vector space over a finite field,
then $\spn:\fin(A)\to\fin(A)$ is a non-degenerate pregeometry on $A$.

Let $S,T\in\fin(A)$. $S$ is \emph{independent over} $T$ if $a\notin\cl(T\cup(S\setminus\{a\}))$ for all $a\in S$.
$S$ is \emph{independent} if it is independent over $\varnothing$.
By the exchange property, $\{a_i\mid i<n\}$ is independent over $T$
if and only if $a_j\notin\cl(T\cup\{a_i\mid i<j\})$ for all $j<n$.

\begin{lemma}\label{sh23}
Let $\cl$ be a pregeometry on an infinite set $A$ and let $S,T\in\fin(A)$ be independent.
If $|S|=|T|$, then $|\cl(S)|=|\cl(T)|$.
\end{lemma}
\begin{proof}
Without loss of generality, assume that $S\cap T=\varnothing$ and $S\cup T$ is independent.
Suppose $S=\{a_i\mid i<n\}$ and $T=\{b_i\mid i<n\}$.
We prove by induction on $j\leqslant n$ that $|\cl(S)|=|\cl(\{a_i\mid i<n-j\}\cup\{b_i\mid n-j\leqslant i<n\})|$ as follows.
The base case of the induction is trivial. Assume the conclusion holds for $j<n$. To establish the step for $j+1$,
it suffices to show that
\begin{equation}\label{sh24}
|\cl(X\cup\{a_{n-j-1}\})|=|\cl(X\cup\{b_{n-j-1}\})|,
\end{equation}
where
\[
X=\{a_i\mid i<n-j-1\}\cup\{b_i\mid n-j\leqslant i<n\}.
\]
Since $a_{n-j-1},b_{n-j-1}\in\cl(S\cup T)\setminus\cl(X)$,
by the local homogeneity, there is a permutation $\pi$ of $\cl(S\cup T)$ that preserves $\cl$ on $\fin(\cl(S\cup T))$,
fixes $\cl(X)$ pointwise, and takes $a_{n-j-1}$ to $b_{n-j-1}$.
Since $\pi[X\cup\{a_{n-j-1}\}]=X\cup\{b_{n-j-1}\}$, it follows that $\pi[\cl(X\cup\{a_{n-j-1}\})]=\cl(X\cup\{b_{n-j-1}\})$,
from which \eqref{sh24} follows. Finally, by setting $j=n$, we obtain $|\cl(S)|=|\cl(T)|$.
\end{proof}

\begin{theorem}\label{sh26}
For every amorphous set $A$ of projective type, $\fin(A)$, and consequently $\mathscr{P}(A)$, is dually Dedekind infinite.
\end{theorem}
\begin{proof}
Let $\cl$ be a non-degenerate pregeometry on $A$.
Then there exists a non-empty $E\in\fin(A)$ with $|E|$ minimal such that $\cl(E)\neq\bigcup_{a\in E}\cl(\{a\})$.
Then $E$ must be independent. We claim that for every independent $X\in\fin(A)$ with $|X|=|E|$,
\begin{equation}\label{sh25}
\cl(X)\neq\bigcup_{a\in X}\cl(\{a\}).
\end{equation}
In fact, since $E$ is independent, it follows that $E\subseteq A\setminus\cl(\varnothing)$.
By the exchange property, the sets $\cl(\{a\})\setminus\cl(\varnothing)$ ($a\in E$) are pairwise disjoint. Hence,
\[
|\cl(\varnothing)|+\sum_{a\in E}|\cl(\{a\})\setminus\cl(\varnothing)|=\left|\bigcup_{a\in E}\cl(\{a\})\right|<|\cl(E)|.
\]
For similar reasons, it follows from Lemma~\ref{sh23} that
\begin{align*}
\left|\bigcup_{a\in X}\cl(\{a\})\right|
& =|\cl(\varnothing)|+\sum_{a\in X}|\cl(\{a\})\setminus\cl(\varnothing)| \\
& =|\cl(\varnothing)|+\sum_{a\in E}|\cl(\{a\})\setminus\cl(\varnothing)|<|\cl(E)|=|\cl(X)|,
\end{align*}
and hence \eqref{sh25} follows.

Clearly, $|E|\geqslant2$. Let $D$ be a subset of $E$ obtained by removing two points.
Let $\mathcal{C}=\{W\in\fin(A)\mid\cl(D\cup W)=W\}$.
We construct a non-injective surjection $f:\fin(A)\to\fin(A)$ as follows. For each $S\in\fin(A)$, define
\[
f(S)=
\begin{cases}
S\setminus\bigcup\{W\in\mathcal{C}\mid W\setminus\cl(D)\subseteq S\text{ with }|W|\text{ maximal}\} & \text{if $S\cap\cl(D)=\varnothing$,}\\
S & \text{otherwise.}
\end{cases}
\]
Note that $f$ is not injective since $f(W\setminus\cl(D))=\varnothing$ for all $W\in\mathcal{C}$.

We show that $f$ is surjective as follows. Let $T\in\fin(A)$. If $T\cap\cl(D)\neq\varnothing$, then $T=f(T)\in\ran(f)$.
Suppose $T\cap\cl(D)=\varnothing$. We need to construct an $S\in\fin(A)$ such that $f(S)=T$.
Let $n=|T|$. We define by recursion a sequence $\langle a_i\rangle_{i\leqslant n}$ of elements of $A$ as follows.
Let $j\leqslant n$ and assume that $\langle a_i\rangle_{i<j}$ has already been defined. Since $A$ is infinite, we can choose
\[
a_j\in A\setminus\cl(D\cup T\cup\{a_i\mid i<j\}).
\]
Then $\{a_i\mid i\leqslant n\}$ is independent over $D\cup T$.
Let $U=\cl(D\cup\{a_i\mid i\leqslant n\})$ and let $S=T\cup(U\setminus\cl(D))$.
It is easy to see that $\cl(D\cup T)\cap U=\cl(D)$.

We conclude the proof by showing $f(S)=T$. By definition of $f$, it suffices to show that whenever
$W\in\mathcal{C}$ satisfies $W\setminus\cl(D)\subseteq S$ and $|W|\geqslant|U|$, we have $W=U$.
Assume towards a contradiction that there exists $W\in\mathcal{C}$ such that $W\setminus\cl(D)\subseteq S$,
$|W|\geqslant|U|$, and $W\neq U$. Then $W\subseteq T\cup U$, and therefore $W\cap T\neq\varnothing$.
Since $|W|\geqslant|U|\geqslant|\cl(D)|+n+1>|\cl(D)\cup T|$, it follows that $W\cap(U\setminus\cl(D))\neq\varnothing$.
Let $b\in W\cap T$ and let $c\in W\cap(U\setminus\cl(D))$.
Let $X=D\cup\{b,c\}$. $X$ is independent and $|X|=|E|$.
By \eqref{sh25}, there exists $d\in\cl(X)\setminus\bigcup_{a\in X}\cl(\{a\})$.
By the minimality of $|E|$, $d\notin\cl(D\cup\{b\})$ and $d\notin\cl(D\cup\{c\})$.
Now, if $d\in T$, then by the exchange property we have $c\in\cl(D\cup T)$;
and if $d\in U$, then by the exchange property we have $b\in U$.
In either case, we obtain a contradiction to $\cl(D\cup T)\cap U=\cl(D)$.
\end{proof}

\section{Strictly amorphous sets and strongly amorphous sets}
In this section, we show that strictly amorphous sets and strongly amorphous sets coincide.
Let $\mathscr{L}_{\doteq}$ denote the first-order language of equality.
For each set $E$, let $\mathscr{L}_{\doteq}^E$ be the extension of $\mathscr{L}_{\doteq}$
obtained by adding constant symbols $\dot{e}$ for each $e\in E$.

\begin{lemma}\label{sh31}
Let $A$ be a strictly amorphous set. For every positive integer $k$ and every $R\subseteq A^k$,
there is a finite subset $E_R$ of $A$ such that for some quantifier-free
$\mathscr{L}_{\doteq}^{E_R}$-formula $\phi(x_1,\dots,x_k)$ we have
\[
R=\bigl\{\langle a_1,\dots,a_k\rangle\in A^k\bigm|A\models\phi[a_1,\dots,a_k]\bigr\}.
\]
\end{lemma}
\begin{proof}
We define $E_R$ for each $R\subseteq A^k$ by recursion on $k$ as follows.

Let $R\subseteq A$. Then $R$ is finite or cofinite since $A$ is amorphous. Define
\[
E_R=
\begin{cases}
R            & \text{if $R$ is finite,}\\
A\setminus R & \text{otherwise.}
\end{cases}
\]
Clearly, if $E_R=\{e_1,\dots,e_n\}$, then the quantifier-free $\mathscr{L}_{\doteq}^{E_R}$-formula
\[
\phi(x)=
\begin{cases}
\phantom{\neg(}x\doteq\dot{e}_1\vee\dots\vee x\doteq\dot{e}_n & \text{if $R$ is finite,}\\
\neg(x\doteq\dot{e}_1\vee\dots\vee x\doteq\dot{e}_n)          & \text{otherwise,}
\end{cases}
\]
satisfies $R=\{a\in A\mid A\models\phi[a]\}$.

Let $k$ be a positive integer and let $R\subseteq A^{k+1}$. For each $a\in A$, let
\[
R_a=\{\langle a_1,\dots,a_k\rangle\in A^k\mid\langle a,a_1,\dots,a_k\rangle\in R\}.
\]
By the induction hypothesis, for each $a\in A$, the set $E_{R_a}$ is defined so that the required condition holds.
For each $b\in A$, define by recursion
\begin{align*}
F_{b,0}   & =\{b\}, \\
F_{b,n+1} & =F_{b,n}\cup\bigcup_{a\in F_{b,n}}E_{R_a},\\
F_b       & =\bigcup_{n\in\omega}F_{b,n}.
\end{align*}
Clearly, each $F_{b,n}$ is finite. Since $A$ is amorphous, there are no surjections from $A$ onto $\omega$,
and hence there exists an $l\in\omega$ such that $F_b=F_{b,l}$ is finite.
Again, since $A$ is amorphous, there is a unique $m\in\omega$ such that the set
\[
B=\{b\in A\mid|F_b|=m\}
\]
is cofinite.

Consider
\[
\pi=\{F_b\cap B\mid b\in B\}\cup\{\{b\}\mid b\in A\setminus B\}.
\]
We claim that $\pi$ is a partition of $A$. Let $b,c\in B$ be such that $(F_b\cap B)\cap(F_c\cap B)\neq\varnothing$.
Let $d\in F_b\cap F_c\cap B$. An easy induction shows that $F_{d,n}\subseteq F_b\cap F_c$ for every $n\in\omega$, so $F_d\subseteq F_b\cap F_c$.
Since $b,c,d\in B$, $|F_b|=|F_c|=|F_d|=m$, and hence $F_b=F_c=F_d$, which implies that $F_b\cap B=F_c\cap B$.

Since $A$ is strictly amorphous, all but finitely many blocks of $\pi$ are singletons. Let
\[
C=\{b\in B\mid F_b\cap B\text{ is a singleton}\}.
\]
Then $C$ is cofinite, and for every $a\in C$, $F_a\cap B=\{a\}$, so $E_{R_a}\subseteq(A\setminus B)\cup\{a\}$,
which implies that there exists a quantifier-free $\mathscr{L}_{\doteq}^{A\setminus B}$-formula $\psi(x,x_1,\dots,x_k)$ such that
\[
R_a=\bigl\{\langle a_1,\dots,a_k\rangle\in A^k\bigm|A\models\psi[a,a_1,\dots,a_k]\bigr\};
\]
let $\Gamma_a$ be the set of all such $\psi$. Let $\sim$ be the equivalence relation on $C$ defined by
\[
a\sim b\qquad\text{if and only if}\qquad\Gamma_a=\Gamma_b.
\]
Since there are only finitely many inequivalent quantifier-free $\mathscr{L}_{\doteq}^{A\setminus B}$-formulas
with free variables among $x,x_1,\dots,x_k$, it follows that there are only finitely many $\sim$-equivalence classes.
Since $A$ is amorphous, there is a unique $\sim$-equivalence class $D$ which is cofinite. Finally, we define
\[
E_R=\bigcup_{b\in A\setminus D}F_b.
\]

We conclude the proof by showing that $E_R$ has the required properties. Suppose that $E_R=\{e_1,\dots,e_n\}$.
For every $i=1,\dots,n$, since $E_{R_{e_i}}\subseteq E_R$, by the induction hypothesis,
there exists a quantifier-free $\mathscr{L}_{\doteq}^{E_R}$-formula $\phi_i(x_1,\dots,x_k)$ such that
\[
R_{e_i}=\bigl\{\langle a_1,\dots,a_k\rangle\in A^k\bigm|A\models\phi_i[a_1,\dots,a_k]\bigr\}.
\]
Since $A\setminus B\subseteq A\setminus D\subseteq E_R$ and $D$ is a $\sim$-equivalence class,
it follows that there exists a quantifier-free $\mathscr{L}_{\doteq}^{E_R}$-formula $\psi(x,x_1,\dots,x_k)$
such that for every $a\in A\setminus E_R$,
\[
R_a=\bigl\{\langle a_1,\dots,a_k\rangle\in A^k\bigm|A\models\psi[a,a_1,\dots,a_k]\bigr\}.
\]
Hence, the quantifier-free $\mathscr{L}_{\doteq}^{E_R}$-formula
\[
\phi(x,x_1,\dots,x_k)=\bigvee_{i=1}^n(x\doteq\dot{e}_i\wedge\phi_i)\vee\Biggl(\bigwedge_{i=1}^n\neg(x\doteq\dot{e}_i)\wedge\psi\Biggr)
\]
satisfies
\[
R=\bigl\{\langle a,a_1,\dots,a_k\rangle\in A^{k+1}\bigm|A\models\phi[a,a_1,\dots,a_k]\bigr\}.\qedhere
\]
\end{proof}

\begin{theorem}\label{sh32}
A set $A$ is strictly amorphous if and only if it is strongly amorphous.
\end{theorem}
\begin{proof}
Suppose that $A$ is strictly amorphous. By Lemma~\ref{sh31}, for every positive integer $k$ and every $k$-ary relation $R$ on $A$,
there is a quantifier-free $\mathscr{L}_{\doteq}$-formula $\phi(x_1,\dots,x_k,y_1,\dots,y_n)$ such that
\[
R=\bigl\{\langle a_1,\dots,a_k\rangle\in A^k\bigm|A\models\phi[a_1,\dots,a_k,e_1,\dots,e_n]\bigr\}
\]
for some $e_1,\dots,e_n\in A$. Hence, $A$ is strongly amorphous.

For the other direction, suppose that $A$ is strongly amorphous.
Every subset of $A$ is first-order definable from finitely many parameters in $\mathscr{L}_{\doteq}$,
and hence is either finite or cofinite. Therefore, $A$ is amorphous. Let $\pi$ be an arbitrary partition of $A$.
Since the equivalence relation induced by $\pi$ is first-order definable from a finite set $E$ of parameters in $\mathscr{L}_{\doteq}$,
it follows that either all elements of $A\setminus E$ lie in a single block of $\pi$,
or each element of $A\setminus E$ forms a singleton block of $\pi$.
Hence, all but finitely many blocks of $\pi$ are singletons. Thus, $A$ is strictly amorphous.
\end{proof}

\section{Strictly amorphous sets and dual Dedekind finiteness}
In this section, we show that every finite power of the power set of a strictly amorphous set is dually Dedekind finite.
A set $A$ is \emph{power Dedekind finite} if the power set of $A$ is Dedekind finite; otherwise, $A$ is \emph{power Dedekind infinite}.
According to Kuratowski's celebrated theorem (see~\cite[Proposition~5.3]{Halbeisen2025}),
a set $A$ is power Dedekind infinite if and only if there exists a surjection from $A$ onto $\omega$.
The class of power Dedekind finite sets is closed under finite products (see~\cite[Theorem~1(vi)]{Truss1974}).
Clearly, every amorphous set is power Dedekind finite, and so is every finite power of an amorphous set.

\begin{lemma}\label{sh41}
Let $n\in\omega$. For each $i<n$, let $e_i\in\omega$ and let $X_i=\bigcup_{l\in\omega}X_{i,l}$, where $X_{i,l}$ is
power Dedekind finite for all $l\in\omega$. If $X=\prod_{i<n}X_i^{e_i}$ is dually Dedekind infinite and $o\notin X$,
then there exist a surjection $g:X\to\{o\}\cup X$ and functions $h,p_{i,j}:\omega\to\omega$ \textup{($i<n,j<e_i$)} such that
\begin{enumerate}[label=\upshape(\arabic*)]
  \item $h$ is increasing,
  \item every $p_{i,j}$ is monotonic, and at least one of $p_{i,j}$ is unbounded,
  \item for every $k\in\omega$, $g^{-h(k)-1}[\{o\}]\subseteq\prod_{i<n}\prod_{j<e_i}X_{i,p_{i,j}(k)}$.
\end{enumerate}
\end{lemma}
\begin{proof}
See~\cite[Lemma~2.4]{Mao2025}.
\end{proof}

We write $[A]^k$ for the set of $k$-element subsets of $A$.
For a tuple $\bar{a}=\langle a_1,\dots,a_k\rangle$, we write $\{\bar{a}\}$ for the set $\{a_1,\dots,a_k\}$.

\begin{corollary}\label{sh42}
Let $A$ be a power Dedekind finite set and let $n\in\omega$. If $\fin(A)^n$ is dually Dedekind infinite,
then there exist a surjection $g:\fin(A)^n\to\{\varnothing\}\cup\fin(A)^n$ and functions $h,p_i:\omega\to\omega$ \textup{($i<n$)} such that
\begin{enumerate}[label=\upshape(\arabic*)]
  \item $h$ is increasing,
  \item every $p_i$ is monotonic, and at least one of $p_i$ is unbounded,
  \item for every $m\in\omega$, $g^{-h(m)-1}[\{\varnothing\}]\subseteq\prod_{i<n}[A]^{p_i(m)}$.
\end{enumerate}
\end{corollary}
\begin{proof}
In Lemma~\ref{sh41}, take $o=\varnothing$, $e_i=1$, $X_i=\fin(A)$, and $X_{i,l}=[A]^l$ for all $i<n$ and all $l\in\omega$.
\end{proof}

\begin{theorem}\label{sh43}
For all strictly amorphous sets $A$ and all natural numbers $n$,
$\fin(A)^n$, and consequently $\mathscr{P}(A)^n$, is dually Dedekind finite.
\end{theorem}
\begin{proof}
Since $A$ is amorphous, $\mathscr{P}(A)$ is equinumerous to $2\times\fin(A)$,
which implies that $\mathscr{P}(A)^n$ is equinumerous to $2^n\times\fin(A)^n$.
Since the class of dually Dedekind finite sets is closed under finite unions (see~\cite[Theorem~1(iv)]{Truss1974}),
it follows that $\fin(A)^n$ is dually Dedekind finite if and only if $\mathscr{P}(A)^n$ is.
We prove that $\fin(A)^n$ is dually Dedekind finite as follows.

We assume the contrary and aim for a contradiction.
By Corollary~\ref{sh42}, there exist a surjection $g:\fin(A)^n\to\{\varnothing\}\cup\fin(A)^n$
and functions $h,p_i:\omega\to\omega$ ($i<n$) such that
\begin{enumerate}[label=\upshape(\arabic*)]
  \item $h$ is increasing,
  \item every $p_i$ is monotonic, and at least one of $p_i$ is unbounded,
  \item for every $m\in\omega$, $g^{-h(m)-1}[\{\varnothing\}]\subseteq\prod_{i<n}[A]^{p_i(m)}$.
\end{enumerate}
Without loss of generality, assume that $p_0$ is unbounded.

For each $\bar{k}=\langle k_0,\dots,k_{n-1}\rangle,\bar{l}=\langle l_0,\dots,l_{n-1}\rangle\in\omega^n\setminus\{\bar{0}\}$
and each $m\in\omega\setminus\{0\}$, let
\[
R_{\bar{k},\bar{l},m}=\left\{\langle\bar{a}_0,\dots,\bar{a}_{n-1},\bar{b}_0,\dots,\bar{b}_{n-1}\rangle\middle|
\begin{array}{l}
\bar{a}_i\in A^{k_i}\text{ and }\bar{b}_i\in A^{l_i}\text{ for every $i<n$, and}\\
g^m(\{\bar{a}_0\},\dots,\{\bar{a}_{n-1}\})=\langle\{\bar{b}_0\},\dots,\{\bar{b}_{n-1}\}\rangle
\end{array}
\right\};
\]
by Lemma~\ref{sh31}, we can define $E_{\bar{k},\bar{l},m}\in\fin(A)$ such that for some
$\mathscr{L}_{\doteq}^{E_{\bar{k},\bar{l},m}}$-formula $\phi(\bar{x}_0,\dots,\bar{x}_{n-1},\bar{y}_0,\dots,\bar{y}_{n-1})$ we have
\begin{equation}\label{sh44}
R_{\bar{k},\bar{l},m}=\bigl\{\langle\bar{a}_0,\dots,\bar{a}_{n-1},\bar{b}_0,\dots,\bar{b}_{n-1}\rangle
\bigm|A\models\phi[\bar{a}_0,\dots,\bar{a}_{n-1},\bar{b}_0,\dots,\bar{b}_{n-1}]\bigr\}.
\end{equation}
Note that $\{E_{\bar{k},\bar{l},m}\mid\bar{k},\bar{l}\in\omega^n\setminus\{\bar{0}\},m\in\omega\setminus\{0\}\}$
is a countable subset of $\fin(A)$, and hence it is finite, since $\fin(A)$ is Dedekind finite. Let
\[
E=\bigcup\{E_{\bar{k},\bar{l},m}\mid\bar{k},\bar{l}\in\omega^n\setminus\{\bar{0}\},m\in\omega\setminus\{0\}\}.
\]
Then $E$ is a finite subset of $A$ such that, for each $\bar{k},\bar{l}\in\omega^n\setminus\{\bar{0}\}$ and each $m\in\omega\setminus\{0\}$,
there exists an $\mathscr{L}_{\doteq}^E$-formula $\phi(\bar{x}_0,\dots,\bar{x}_{n-1},\bar{y}_0,\dots,\bar{y}_{n-1})$ that satisfies \eqref{sh44}.

Since $p_0$ is unbounded, there exists an increasing sequence $\langle m_j\rangle_{j\in\omega}$
such that $0<p_0(m_j)<p_0(m_{j+1})$ for all $j\in\omega$.
Since $g$ is surjective, there exists a finite sequence $\langle\bar{T}_j\rangle_{j\leqslant2^{2^{n+|E|}}}$
of elements of $\fin(A)^n$ such that, for all $j<2^{2^{n+|E|}}$,
\begin{align*}
g^{h(m_0)+1}(\bar{T}_0)              & =\varnothing, \\
g^{h(m_{j+1})-h(m_j)}(\bar{T}_{j+1}) & =\bar{T}_j.
\end{align*}
Since $g^{-h(m)-1}[\{\varnothing\}]\subseteq\prod_{i<n}[A]^{p_i(m)}$ for all $m\in\omega$,
we have $\bar{T}_j\in\prod_{i<n}[A]^{p_i(m_j)}$ for all $j\leqslant2^{2^{n+|E|}}$.

Let $\bar{S}=\bar{T}_{2^{2^{n+|E|}}}=\langle S_0,\dots,S_{n-1}\rangle$,
and let $\sim$ be the equivalence relation on $A$ defined by
\[
a\sim b\iff\text{either }a=b\in E\text{ or }a,b\notin E\text{ and }\forall i<n(a\in S_i\leftrightarrow b\in S_i).
\]
Clearly, the number of finite equivalence classes of $\sim$ is less than $2^{n+|E|}$.
For each $j<2^{2^{n+|E|}}$, let $\bar{T}_j=\langle T_{j,0},\dots,T_{j,n-1}\rangle$.
Since $|T_{j,0}|=p_0(m_j)$ and $\langle p_0(m_j)\rangle_{j\in\omega}$ is increasing,
it follows that there exists an $r<2^{2^{n+|E|}}$ such that $T_{r,0}$ is not a union of equivalence classes of $\sim$;
that is, there are $c,d\in A$ such that $c\sim d$, $c\in T_{r,0}$ and $d\notin T_{r,0}$. Clearly, $c,d\notin E$.
Let $\pi$ be the transposition that interchanges $c$ and $d$.
Then $\pi$ fixes $E$ pointwise, $\pi[S_i]=S_i$ for all $i<n$, and $\pi[T_{r,0}]\neq T_{r,0}$.

For each $i<n$, let $k_i=|S_i|$ and $l_i=|T_{r,i}|$, and let $m=h(m_{2^{2^{n+|E|}}})-h(m_r)$.
Then $g^m(\bar{S})=\bar{T}_r$. For each $i<n$, let $\bar{a}_i\in A^{k_i}$ and $\bar{b}_i\in A^{l_i}$
be such that $S_i=\{\bar{a}_i\}$ and $T_{r,i}=\{\bar{b}_i\}$, respectively. Then
\[
\langle\bar{a}_0,\dots,\bar{a}_{n-1},\bar{b}_0,\dots,\bar{b}_{n-1}\rangle\in R_{\bar{k},\bar{l},m}.
\]
Since isomorphisms preserve all formulas, it follows from \eqref{sh44} that
\[
\langle\pi\bar{a}_0,\dots,\pi\bar{a}_{n-1},\pi\bar{b}_0,\dots,\pi\bar{b}_{n-1}\rangle\in R_{\bar{k},\bar{l},m}.
\]
Hence, $g^m(\bar{S})=\langle\pi[T_{r,0}],\dots,\pi[T_{r,n-1}]\rangle$, contradicting $\pi[T_{r,0}]\neq T_{r,0}$.
\end{proof}

\section{Open questions}
We conclude the article with the following two questions.
First, we do not know whether the converse of Theorem~\ref{sh26} holds.

\begin{question}
Does $\mathsf{ZF}$ prove that every amorphous set having a dually Dedekind infinite power set is of projective type?
\end{question}

Second, we wonder whether the consistency result in \cite{Mao2025} can be generalized to the power sets of amorphous sets.

\begin{question}
Is it consistent with $\mathsf{ZF}$ that there exists a family $\langle A_n \rangle_{n \in \omega}$ of amorphous sets such that,
for all $n\in\omega$, $\mathscr{P}(A_n)^n$ is dually Dedekind finite, whereas $\mathscr{P}(A_n)^{n+1}$ is dually Dedekind infinite?
\end{question}


\begin{thebibliography}{99}

\bibitem{Dedekind1888} R.~Dedekind,
\emph{Was sind und was sollen die Zahlen},
Vieweg, Braunschweig, 1888.

\bibitem{Degen1994} J.~W.~Degen,
Some aspects and examples of infinity notions,
\emph{Math. Log. Q.} \textbf{40} (1994), 111--124.

\bibitem{Goldstern1997} M.~Goldstern,
Strongly amorphous sets and dual Dedekind infinity,
\emph{Math. Log. Q.} \textbf{43} (1997), 39--44.

\bibitem{Halbeisen2025} L.~Halbeisen,
\emph{Combinatorial Set Theory: With a Gentle Introduction to Forcing},
3rd ed., Springer Monogr. Math., Springer, Cham, 2025.

\bibitem{Herrlich2015} H.~Herrlich, P.~Howard, and E.~Tachtsis,
On a certain notion of finite and a finiteness class in set theory without choice,
\emph{Bull. Pol. Acad. Sci. Math.} \textbf{63} (2015), 89--112.

\bibitem{Hodges1993} W.~Hodges,
\emph{Model theory},
Encyclopedia Math. Appl. \textbf{42}, Cambridge Univ. Press, Cambridge, 1993.

\bibitem{Jech1973} T.~Jech,
\emph{The Axiom of Choice},
Stud. Logic Found. Math. \textbf{75}, North-Holland, Amsterdam, 1973.

\bibitem{Levy1958} A.~Levy,
The independence of various definitions of finiteness,
\emph{Fund. Math.} \textbf{46} (1958), 1--13.

\bibitem{Mao2025} R.~Mao and G.~Shen,
A note on dual Dedekind finiteness,
\emph{Log. J. IGPL} \textbf{33} (2025), jzaf069.

\bibitem{Panasawatwong2025} S.~Panasawatwong and J.~Truss,
Dedekind-finite cardinals having countable partitions,
\emph{J. Symb. Log.} \textbf{90} (2025), 1308--1323.

\bibitem{Shen2021} G.~Shen,
A choice-free cardinal equality,
\emph{Notre Dame J. Form. Log.} \textbf{62}, 577--587 (2021).

\bibitem{Spivsiak1993} L.~Spi\v siak,
Dependences between definitions of finiteness. II,
\emph{Czechoslovak Math. J.} \textbf{43} (1993), 391--407.

\bibitem{Truss1974} J.~Truss,
Classes of Dedekind finite cardinals,
\emph{Fund. Math.} \textbf{84} (1974), 187--208.

\bibitem{Truss1995} J.~Truss,
The structure of amorphous sets,
\emph{Ann. Pure Appl. Logic} \textbf{73} (1995), 191--233.

\end{thebibliography}
\end{document}